\documentclass{article}

\author{Ofir Gorodetsky}

\newcommand{\Addresses}{{
		\bigskip
		\footnotesize
		
		\textsc{Mathematical Institute, University of Oxford, Oxford, OX2 6GG, UK}\par\nopagebreak
		\textit{E-mail address:} \texttt{ofir.goro@gmail.com}
}}

\title{New representations for all sporadic Ap{\'e}ry-like sequences, with applications to congruences}
\date{}

\usepackage[margin=1in]{geometry}
\usepackage[all]{xy}
\usepackage{amssymb}
\usepackage{amsfonts}
\usepackage{amsthm}
\usepackage{amsmath}
\usepackage{hyperref}
\usepackage{mathtools}
\usepackage{url}
\usepackage{xcolor}
\usepackage{tikz}

\newtheorem*{thm*}{Theorem}
\newtheorem{thm}{Theorem}[section]

\newtheorem{lem}[thm]{Lemma}  
\newtheorem{proposition}[thm]{Proposition}
\newtheorem{cor}[thm]{Corollary}

\theoremstyle{definition}
\theoremstyle{remark}

\theoremstyle{definition}

\newcommand{\QQ}{\mathbb{Q}}
\newcommand{\CC}{\mathbb{C}}
\newcommand{\ZZ}{\mathbb{Z}}

\newcommand{\AZ}{AZ}
\newcommand{\CT}{\mathrm{CT}}
\newcommand{\CTS}{\mathrm{CTS}}

\mathchardef\mhyphen="2D

\newtheorem{remark}[thm]{Remark}
\newtheorem{definition}[thm]{Definition}

\numberwithin{equation}{section}

\newtheorem*{question*}{Question}

\mathtoolsset{showonlyrefs}

\begin{document}
\maketitle

\begin{abstract}
We find new representations, in terms of constant terms of powers of Laurent polynomials, for all the 15 sporadic Ap{\'e}ry-like sequences discovered by Zagier, Almkvist-Zudilin and Cooper.

The new representations lead to binomial expressions for the sequences, which, as opposed to previous expressions, do not involve powers of 3 or 8. We use these to establish the supercongruence
$B_{np^k} \equiv B_{np^{k-1}} \bmod p^{2k}$
for all primes $p \ge 3$ and integers $n,k \ge 1$, where $B_n$ is a sequence discovered by Zagier, known as Sequence $\mathbf{B}$.

Additionally, for 14 of the 15 sequences, the Newton polytopes of the Laurent polynomials contain the origin as their only interior integral point. This property allows us to prove that these sequences satisfy a strong form of the Lucas congruences, extending work of Malik and Straub. Moreover, we obtain lower bounds on the $p$-adic valuation of these sequences via recent work of Delaygue.
\end{abstract}

\section{Introduction}
Roger Ap{\'e}ry's ingenious proof of the irrationality of $\zeta(3)=\sum_{n \ge 1} 1/n^3$ and $\zeta(2)=\sum_{n \ge 1} 1/n^2$ \cite{apery1979, apery1981} utilized two particular sequences:
\begin{equation}\label{eq:binom sums}
a_n =\sum_{k=0}^{n} \binom{n}{k}^2 \binom{n+k}{k}^2,  \, b_n=\sum_{k=0}^{n} \binom{n}{k}^2 \binom{n+k}{k}
\end{equation}
for $n \ge 0$. These sequences are referred to as Ap{\'e}ry sequences or Ap{\'e}ry numbers. Ap{\'e}ry constructed explicit rational approximations to $\zeta(3)$ and $\zeta(2)$, with denominators being small multiples of $a_n$ and $b_n$, respectively. These rational approximations turn out to converge quickly enough (in a particular sense) to $\zeta(3)$ and $\zeta(2)$ as $n \to \infty$ in order to conclude that $\zeta(3)$, $\zeta(2) \notin \QQ$.

Originally, Ap{\'e}ry defined $a_n$ and $b_n$ through 3-term recurrences with given initial values:
\begin{equation}\label{eq:recs}
\begin{split}
&(n+1)^3 a_{n+1} - (2n+1)(17n^2+17n+5)a_n+n^3 a_{n-1}=0,\qquad (a_0=1,\, a_1=5)\\
&(n+1)^2 b_{n+1} -(11 n^2+11 n+3) b_n -n^2 b_{n-1}=0,\qquad (b_0=1,\, b_1=3)
\end{split}
\end{equation}
and showed that $a_n$ and $b_n$ can be expressed as the binomial sums in \eqref{eq:binom sums}. The recurrences for $a_n$ and $b_n$ can be restated in terms of differential equations for the g.f.s (generating functions)
\begin{equation}
 F_{\mathbf{a}}(z) = \sum_{n \ge 0} a_n z^n \mbox{ and } F_{\mathbf{b}}(z) = \sum_{n \ge 0} b_n z^n,
\end{equation}
namely, setting $\theta := z \frac{d}{dz}$, $F_{\mathbf{a}}$ is a solution to the third-order equation
\begin{equation}\label{eq:ode1}
(\theta^3-z(2\theta+1)(17\theta^2 +17\theta +5)+z^2(\theta+1)^3)y=0,
\end{equation}
while $F_{\mathbf{b}}$ solves the second-order equation
\begin{equation}\label{eq:ode2}
(\theta^2-z(11\theta^2 +11\theta +3)-z^2(\theta+1)^2)y=0.
\end{equation}
Beukers \cite{beukers1987,beukers19872} showed that $F_{\mathbf{a}}$, $F_{\mathbf{b}}$ enjoy modular parametrization. Namely, there is a modular function $t_1(z)$ ($z$ in the upper-half plane) with respect to the modular subgroup $\Gamma_1(6)$ such that $F_{\mathbf{a}} \circ t_1$ is a modular form of weight 2 with respect to $\Gamma_1(6)$. Similarly, there is a modular function $t_2(z)$ with respect to $\Gamma_1(5)$ such that $F_{\mathbf{b}} \circ t_2$ is a modular form of weight 1 with respect to $\Gamma_1(5)$.

This is not a coincidence, and in fact related to the shape of the specific differential equations. Equation \eqref{eq:ode2} is a Picard-Fuchs equation, while \eqref{eq:ode1} is a symmetric square of a Picard-Fuchs equation. We shall not discuss this geometric aspect here, and instead refer the reader to the papers \cite{stienstra1985,peters1989,verrill2001} and to Zagier's survey \cite{zagier2018}.
\subsection{Sporadic sequences}
There is no apparent reason for solutions of the recurrences \eqref{eq:recs} to be integers (or, equivalently, for solutions of \eqref{eq:ode1} or \eqref{eq:ode2}, which are regular at the origin, to have only integral coefficients in their Taylor expansion around $z=0$). It is a surprise that $a_n$ and $b_n$ are integers. For this reason, Zagier \cite{zagier2009} searched for triples $(A,B,\lambda) \in \ZZ^3$ such that the recurrence
\begin{equation}\label{eq:zag eq}
(n+1)^2 u_{n+1} -(An^2+An+\lambda)u_n+Bn^2u_{n-1}=0, \qquad n \ge -1,
\end{equation}
has a solution in \emph{integers} with $u_{-1}=0$, $u_0=1$. A solution $u_n$ gives rise to a power series $y=\sum u_n z^n$ satisfying the second-order differential equation
\begin{equation}\label{eq:diff2}
(\theta^2-z(A\theta^2 +A\theta +\lambda)+Bz^2(\theta+1)^2)y=0
\end{equation}
and vice-versa. For $(A,B,\lambda)=(11,-1,3)$ we already know that there is such a desired solution, namely $b_n$. In his computer search, two well-understood families were found: Legendrian and hypergeometric. In addition, six sporadic solutions are found. The two families are easily explained, in the sense that the g.f.s $\sum u_n z^n$ have particularly simple forms, which explain the integrality.  The six sporadic sequences, which include $u_n=b_n$, do not fall into either of these families and are not explained as easily. They are given in the following table.
\begin{center}
	\begin{tabular}{||c c c c ||}
		\hline
		$(A,B,\lambda)$ & Name & Formula & Other names \\ [0.5ex]
		\hline\hline
		$(7,-8,2)$ & $\mathbf{A}$ & $u_n = \sum_{k=0}^{n} \binom{n}{k}^3$ & Franel numbers \\
		$(9,27,3)$ & $\mathbf{B}$ & $u_n = \sum_{k=0}^{\lfloor n/3 \rfloor} (-1)^k 3^{n-3k}\binom{n}{3k} \binom{3k}{2k} \binom{2k}{k}$ &  \\
		$(10,9,3)$ & $\mathbf{C}$ & $u_n = \sum_{k=0}^{n} \binom{n}{k}^2 \binom{2k}{k}$ & \\
		$(11,-1,3)$ & $\mathbf{D}$ & $u_n = \sum_{k=0}^{n} \binom{n}{k}^2 \binom{n+k}{n}$ & Ap{\'e}ry numbers \\
		$(12,32,4)$ & $\mathbf{E}$ & $u_n = \sum_{k=0}^{\lfloor n/2 \rfloor} 4^{n-2k} \binom{n}{2k} \binom{2k}{k}^2$ & \\
		$(17,72,6)$ & $\mathbf{F}$ & $u_n = \sum_{k=0}^{n} (-1)^k 8^{n-k} \binom{n}{k} \sum_{j=0}^{k} \binom{k}{j}^3$ & \\
		[1ex]
		\hline
\end{tabular}
\end{center}
Zagier conjectured that his (finite) search yielded all the integral solutions to \eqref{eq:zag eq}. This is suggested by the fact that so few sporadic solutions were found in a very large parameter domain. As in the case of $b_n$, the differential equations associated with the sequences found by Zagier are Picard–Fuchs equations, and the solutions to the differential equations have modular parametrizations \cite[\S5,~\S7]{zagier2009}.

Similar searches were conducted in relation to $a_n$. Almkvist and Zudilin \cite{almkvist2006} searched for $(a,b,c) \in \ZZ^3$ such that the recurrence
\begin{equation}\label{eq:almkvist}
(n + 1)^3 u_{n+1} - (2n+1)(an^2+an+b)u_n +cn^3 u_{n-1} =0
\end{equation}
has integer solutions with $u_{-1}=0$, $u_0=1$. They found 6 sporadic solutions, including $a_n$, which corresponds to $(a,b,c) = (17,5,1)$:
\begin{center}
	\begin{tabular}{||c c c c ||}
		\hline
		$(a,b,c)$ & Name & Formula & Other names \\ [0.5ex]
		\hline\hline	
		$(7,3,81)$ & $(\delta)$ & $u_n = \sum_{k=0}^{\lfloor n/3 \rfloor} (-1)^k 3^{n-3k} \binom{n}{3k} \binom{n+k}{n} \binom{3k}{2k} \binom{2k}{k}$ & Almkvist-Zudilin numbers\\
		$(11,5,125)$ & $(\eta)$ & $u_n = \sum_{k=0}^{\lfloor n/5 \rfloor} (-1)^k \binom{n}{k}^3 \left( \binom{4n-5k-1}{3n} + \binom{4n-5k}{3n} \right)$ & \\
		$(10,4,64)$ & $(\alpha)$ & $u_n = \sum_{k=0}^{n} \binom{n}{k}^2\binom{2k}{k}\binom{2n-2k}{n-k}$ & Domb numbers \\
		$(12,4,16)$ & $(\epsilon)$ & $u_n = \sum_{k=\lceil n/2 \rceil}^{n} \binom{n}{k}^2 \binom{2k}{n}^2$ &  \\
		$(9,3,-27)$ & $(\zeta)$ & $u_n = \sum_{k,l} \binom{n}{k}^2 \binom{n}{l}\binom{k}{l}\binom{k+l}{n}$ & \\
		$(17,5,1)$ & $(\gamma)$ & $u_n = \sum_{k=0}^{n} \binom{n}{k}^2 \binom{n+k}{k}^2$ &  Ap{\'e}ry numbers\\
		[1ex]
		\hline
	\end{tabular}
\end{center}
In \cite[Thm.~1]{almkvist2011}, Almkvist, van Straten and Zudilin discovered a bijection between Zagier's sequences and Almkvist and Zudilin's (cf. Cooper \cite[p.~400]{cooper2017}). The sequences $\mathbf{A}$-$\mathbf{F}$ correspond to $(\delta)$, $(\zeta)$, $(\alpha)$, $(\eta)$, $(\epsilon)$, $(\gamma)$, in this order, via sending $(A,B,\lambda)$ to $(a,b,c)=(A,A-2\lambda,A^2-4B)$, and the g.f.s for the Almkvist and Zudilin sequences are essentially the squares of the corresponding sequences of Zagier:
\begin{equation*}
\sum_{n\ge 0} u'_n \left(\frac{-x}{1-A x+B x^2}\right)^{n+1} = (-x) \bigg(\sum_{n\ge 0} u_n x^n\bigg)^2 
\end{equation*}
where $u'_n$ is the third-order sequence corresponding to the solution $u_n$ of \eqref{eq:zag eq}.

Later, Cooper \cite{cooper2012} introduced an additional parameter $d \in \ZZ$:
\begin{equation}\label{eq:cooper}
(n + 1)^3 u_{n+1} - (2n+1)(an^2+an+b)u_n +n(cn^2+d)u_{n-1} =0.
\end{equation}
A solution $u_n$ gives rise to a power series $y=\sum u_n z^n$ satisfying the third-order differential equation
\begin{equation}\label{eq:diff3}
(\theta^3-z(2\theta+1)(a\theta^2 +a\theta +b)+z^2(c(\theta+1)^3+d(\theta+1)))y=0
\end{equation}
and vice-versa. Cooper found $3$ additional sporadic solutions, named $s_{7}$, $s_{10}$ and $s_{18}$. He established their modular parametrization (the subscript is related to the level of the corresponding modular form). Formulas for the terms of $s_7$, $s_{10}$ and $s_{18}$ were obtained by Zudilin:
\begin{center}
	\begin{tabular}{||c c c c ||}
		\hline
		$(a,b,c,d)$ & Name & Formula & Other names \\ [0.5ex]
		\hline\hline	
		$(13,4,-27,3)$ & $s_{7}$ & $u_n = \sum_{k=\lceil n/2 \rceil}^{n} \binom{n}{k}^2 \binom{n+k}{k} \binom{2k}{n}$ &  \\
		$(6,2,-64,4)$ & $s_{10}$ & $u_n = \sum_{k=0}^{n} \binom{n}{k}^4$ & Yang-Zudilin numbers \\
		$(14,6,192,-12)$ & $s_{18}$ & $	u_n = \sum_{k=0}^{\lfloor n/3 \rfloor} (-1)^k \binom{n}{k} \binom{2k}{k} \binom{2n-2k}{n-k} \left( \binom{2n-3k-1}{n} + \binom{2n-3k}{n} \right)$ & \\
		[1ex]
		\hline
	\end{tabular}
\end{center}
In total, we have $15$ sporadic solutions, which are often referred to as Ap{\'e}ry-like sequences.
\subsection{Main results}
We say that a sequence $\{u_n \}_{n \ge 1}$ is the constant term sequence of a Laurent polynomial $\Lambda \in \CC[x_1^{\pm 1},x_2^{\pm 1},\ldots,x_d^{\pm 1}]$ if $u_n$ is the constant term of $\Lambda^n$ for every $n \ge 1$. Our main result is
\begin{thm}\label{thm:newton}
Let $\{u_n \}_{n \ge 1}$ be one of the 15 sporadic sequences, and denote by $k \in \{2,3\}$ the order of the differential equation it satisfies. Then $u_n$ is a constant term sequence of a Laurent polynomial $\Lambda$  with integer coefficients in $k$ variables $x_1,\ldots,x_k$. In addition,
\begin{enumerate}
	\item $\Lambda$ may be written as a product of Laurent polynomials, each with coefficients in the set $\{-1,0,1\}$.
	\item $\Lambda$ may be chosen in such a way that it is a (usual) polynomial divided by $\prod_{i=1}^{k} x_i$.
	\item For all but possibly $(\eta)$, $\Lambda$ may be chosen in such a way that its Newton polytope contains the origin as its only interior integral point.
	\item For all but possibly $(\eta)$ and $\mathbf{F}$, $\Lambda$ may be chosen to satisfy the last 3 conditions simultaneously.	In the case of $\mathbf{F}$, one may choose a Laurent polynomial $\Lambda(x_1,x_2)$ satisfying the last two conditions, with $\Lambda(x_1^2,x_2^2)$ satisfying the first condition.
\end{enumerate}

\end{thm}
Specific Laurent polynomials are given in Proposition~\ref{prop:rep}, from which we derive new binomial sum representations for $\mathbf{B}$, $\mathbf{F}$ and $(\delta)$.
\begin{proposition}\label{prop:formulas}
	Consider the set
	\begin{align}
	S(n)  = \left\{(\mathbf{a},\mathbf{b},\mathbf{c}) \in \ZZ_{\ge 0}^9  : \substack{a_1+a_2+a_3 = n,\\b_1+b_2+b_3 = n, \\ c_1+c_2+c_3 = n,} \quad \substack{\forall 1 \le i \le 3:\, a_i + b_i+c_i  =n, \\ 3 \mid b_2+2b_3 +2c_2 + c_3 }\right\}.
	\end{align}
	The sequence $\mathbf{B}$ can be written as
	\begin{equation}
	2(-1)^n B_n = 3\sum_{(\mathbf{a},\mathbf{b},\mathbf{c}) \in S(n)} \binom{n}{\mathbf{a}}\binom{n}{\mathbf{b}} \binom{n}{\mathbf{c}} -\binom{3n}{n,n,n}.
	\end{equation}
	Consider the set
	\begin{align}
	T(n)  = \left\{ (\mathbf{a},\mathbf{b},\mathbf{c},\mathbf{d}, \mathbf{e}) \in \ZZ_{\ge 0}^{15}  : \substack{a_1+a_2+a_3 = n,\\b_1+b_2+b_3 = n, \\ c_1+c_2+c_3 = n,} \quad \substack{ d_1+d_2+d_3=n, \\ e_1+e_2+e_3=n, \\ \forall 1 \le i \le 3:\, a_i + b_i+c_i+d_i+2e_i  =2n} \right\}.
	\end{align}
	The sequence $\mathbf{F}$ can be written as
	\begin{equation}\label{eq:fn form}
	F_n = \sum_{(\mathbf{a},\mathbf{b},\mathbf{c},\mathbf{d}, \mathbf{e}) \in T(n)} \binom{n}{\mathbf{a}}\binom{n}{\mathbf{b}} \binom{n}{\mathbf{c}} \binom{n}{\mathbf{d}} \binom{n}{\mathbf{e}} (-1)^{a_1+b_2+c_3}.
	\end{equation}
	Consider the set
	\begin{align}
	U(n)  = \left\{ (\mathbf{a},\mathbf{b},\mathbf{c},\mathbf{d}) \in \ZZ_{\ge 0}^{12}  : \substack{a_1+a_2+a_3 = n,\\b_1+b_2+b_3 = n, \\ c_1+c_2+c_3 = n,\\d_1+d_2+d_3 = n,} \quad \substack{b_1+c_1+d_1 = n,\\a_1+b_2+d_2 = n,\\ a_2+b_3+c_2 = n}\right\}.
	\end{align}
	The sequence $(\delta)$ can be written as
	\begin{equation}\label{eq:azn form}
	\AZ_n = \sum_{(\mathbf{a},\mathbf{b},\mathbf{c},\mathbf{d}) \in U(n)} \binom{n}{\mathbf{a}}\binom{n}{\mathbf{b}} \binom{n}{\mathbf{c}} \binom{n}{\mathbf{d}}  (-1)^{a_2+b_1+d_3}.
	\end{equation}
\end{proposition}
\subsection{Gauss congruences}
Both $a_n$ and $b_n$ possess a wealth of arithmetic properties. We concentrate on Gauss congruences.
\begin{definition}\label{def:gauss}
	A sequence $\{ u_n\}_{n \ge 1}$ of integers is said to satisfy the Gauss congruences of order $r$ if, for all primes $p \ge r+1$ and all positive integers $k,n$ we have
	\begin{equation}\label{eq:gauss cong}
	u_{np^k} \equiv u_{np^{k-1}} \bmod {p^{rk}}.
	\end{equation}
\end{definition}
The case $r=1$ is simply called `Gauss congruences', while the case $r>1$ is an example of \emph{supercongruences}. 
A well-known example of a sequence satisfying the Gauss congruences is $u_n=a^n$ for any choice of $a \in \ZZ$. Another example of such a sequence is given by $u_n = \binom{2n}{n}$ \cite[Ch.~5.3.3]{robert2000}. In fact, $\binom{2n}{n}$ satisfies the Gauss congruences of order $3$, as shown by Jacobsthal \cite{brun1949} (cf. \cite[Ch.~7.1.6]{robert2000}).

Chowla, Cowles and Cowles \cite{chowla1980}, based on numerical observations, made several conjectures on the values of $a_n$ modulo primes and prime powers. Gessel \cite{gessel1982} proved generalizations of their conjectures. In particular, in \cite[Thm.~3]{gessel1982} he proved that $a_n$ satisfies the Gauss congruences of order $3$, but only in the special case where $k=1$ in \eqref{eq:gauss cong}. This was established independently by Mimura \cite[Prop.~2]{mimura1983}. Finally, in his PhD thesis, Coster \cite{coster1988} showed that the Gauss congruences of order $3$ are satisfied by $a_n$ (without any restriction on $k$), as well as by $b_n$.  

Following numerical observations, it was conjectured that all sporadic sequences satisfy the Gauss congruences of order $2$ or $3$ (depending on the sequence), see \cite[Table~2]{osburn2016}. The fact that they all satisfy the Gauss congruences of order $1$ is explained by modular parametrization, see e.g. \cite{verrill2010,osburn20112,moy2013}. Establishing Gauss congruences of higher order is a harder task. In his thesis, Coster \cite{coster1988} established such congruences for `generalized Ap{\'e}ry numbers', that is 
\begin{equation}
 u^{(r,s,\varepsilon)}_n = \sum_{k=0}^{n} \binom{n}{k}^r \binom{n+k}{k}^s \varepsilon^k
\end{equation}
for $\varepsilon \in \{\pm 1\}$ and $r,s \ge 0$. Specifically, he proved that $ u^{(r,s,\varepsilon)}$ satisfies the Gauss congruences of order 3 as long as $r \ge 2$, or $r=1$, $s \ge 1$ and $\varepsilon=-1$. The sequences $u^{(r,s,\varepsilon)}$ include, as special cases, the sequences $(\gamma)$ and $\mathbf{D}$, as well as $\mathbf{A}$ and $s_{10}$.
Coster's method for proving these congruences was extended and used by various authors:
\begin{itemize}
	\item Osburn and Sahu \cite{osburn2011} proved that $\mathbf{C}$ satisfies the Gauss congruences of order $3$. 
	\item Later, Osburn and Sahu \cite{osburn2013} proved that $(\alpha)$ satisfies the Gauss congruences of order $3$ as well, extending a result of Chan, Cooper and Sica \cite{chan2010}. They also sketched the proof of a similar result for $\mathbf{E}$ (using the representation \eqref{eq:new e}).
	\item Osburn, Sahu and Straub \cite{osburn2016} proved that $s_{18}$ satisfies the Gauss congruences of order 3, as well as a family that includes $s_{7}$, $\mathbf{D}$ and $(\epsilon)$. In \cite[Ex.~3.1]{osburn2016} they sketch a proof of a similar result for $(\eta)$.
\end{itemize}
All these sequences involved only one summation variable, $k$. Recently, the author established that the sequence $(\zeta)$, whose definition involves summation over two variables $k$ and $l$, satisfies the Gauss congruences of order 3 \cite[Cor.~1.4]{gorodetsky2019}. Although the proof uses $q$-congruences, this can be avoided, and one can formulate the proof using Coster's method.

This leaves three sequences for which little is known: $\mathbf{B}$, $\mathbf{F}$ and $(\delta)$. It was proved by Amdeberhan and Tauraso \cite{amdeberhan2016} that $(\delta)$ satisfies the Gauss congruences of order 3, but only in the special case where $k=1$ in \eqref{eq:gauss cong}. The main difficulty in dealing with $\mathbf{B}$, $\mathbf{F}$ and $(\delta)$ has to do with the appearance of powers (of $3$ and $8$) in their definition. As opposed to binomial coefficients, which behave nicely modulo high powers of primes, powers' behavior is much worse. For instance, $u_n=2^n$ satisfies the Gauss congruences of order 1 but not of higher order, as opposed to $u_n = \binom{2n}{n}$. In fact, it is not even known if there are infinitely many primes $p$ such that $2^{p} \equiv 2^1 \bmod {p^2}$; such primes are known as Wieferich primes. We prove the following.
\begin{thm}\label{thm:super}
Sequence $\mathbf{B}$ satisfies the Gauss congruences of order 2.
\end{thm}
This confirms a conjecture of Osburn, Sahu and Straub \cite{osburn2016}. 
\begin{remark}
	Theorem~\ref{thm:newton} also gives a new proof for the fact that all the sporadic sequences are integral and satisfy the Gauss congruences (of order $1$), since these congruences hold for the constant terms of powers of any integral Laurent series.
\end{remark}

\section{Further results}\label{sec:lucas}
\subsection{Lucas and D3 congruences}
 \begin{definition}
	A sequence $\{ u_n\}_{n \ge 1}$ of integers is said to satisfy the Lucas congruences if, for any prime $p$, and any positive integer $n$ with base-$p$ expansion $\sum_{i=0}^{s} n_i p^i$ ($0 \le n_i < p$) we have
	\begin{equation}
	u_{n} \equiv \prod_{i=0}^{s} u_{n_i} \bmod p.
	\end{equation} 
\end{definition}
The name of these congruences comes from a theorem of Lucas \cite{lucas1878}, stating that if $n=\sum_{i=0}^{s} n_i p^i$ and $k=\sum_{i=0}^{s} k_i p^i$ are base-$p$ expansions of $n \ge k \ge0$, then
\begin{equation}
\binom{n}{k} \equiv \prod_{i=0}^{s} \binom{n_i}{k_i} \bmod p.
\end{equation}
Gessel \cite[Thm.~1]{gessel1982}, inspired by the previously mentioned conjectures of Chowla, Cowles and Cowles, has shown that $a_n$ satisfies the Lucas congruences. Deutsch and Sagan \cite[Thm.~5.9]{deutsch2006} showed more generally that for any positive integers $r$ and $s$, the sequence
\begin{equation}
u_n = \sum_{k=0}^{n} \binom{n}{k}^r \binom{n+k}{k}^s
\end{equation}
satisfies the Lucas congruences, which implies that $a_n$ and $b_n$ satisfy the Lucas congruences. They also showed that $u_n=\binom{2n}{n}$ satisfies Lucas congruences \cite[Thm.~4.4]{deutsch2006}. A property stronger than Lucas congruences is the D3 property, named after Dwork \cite{dwork1969}.
\begin{definition}
	Let $\{ u_n\}_{n \ge 1}$ be a sequence of integers, and set $u_0:=1$. We say that $\{u_n\}_{n \ge 1}$ satisfies the D3 congruences with respect to a prime $p$ if
	\begin{equation}\label{eq:d3}
	u_{n+mp^{s}} u_{\lfloor \frac{n}{p} \rfloor} \equiv u_n u_{\lfloor \frac{n+mp^s}{p} \rfloor} \bmod p^s
	\end{equation}
	for all $s,m,n\ge 0$. If this holds for all primes $p$, we simply say that $\{u_n\}_{n \ge 1}$ satisfies the D3 congruences.
\end{definition}
These congruences are implicit in the work of Dwork \cite{dwork1969} and have significance in $p$-adic analysis. We refer the reader to \cite{samol2015,mellit2016} for further information, as well as definitions of related congruences, known as D1 and D2, which together imply D3. One can show that \eqref{eq:d3} with $s=1$ implies the Lucas congruences. Samol and van Straten \cite{samol2015} proved that a certain class of sequences satisfies the D3 congruences. A different proof of this was found by Mellit and Vlasenko \cite{mellit2016}. Before we state their result, we recall that the Newton polytope of a Laurent polynomial $f = \sum_{(a_1,\ldots,a_d) \in \ZZ^d} c_{a_1,\ldots,a_d} x_1^{a_1} \cdots x_d^{a_d} $ is the convex hull of the support of $f$, that is, of $\{ (a_1,\ldots,a_d) \in \ZZ^d : c_{a_1,\ldots,a_d} \neq 0 \}$. We use the notation  $\CT(f)$ for the constant term $c_{0,0,\ldots,0}$ of $f$ (or any Laurent series).
\begin{thm}\label{thm:samol}\cite{samol2015,mellit2016}
	Let $\Lambda \in \ZZ[x_1,x_1^{-1},\ldots,x_d,x_d^{-1}]$ be a Laurent polynomial. If the Newton polytope of $\Lambda$ contains the origin as its only interior integral point, then the sequence $u_n=\CT(\Lambda^n)$ satisfies the D3 congruences.
\end{thm}
As observed e.g. by Mellit and Vlasenko \cite[\S1]{mellit2016}, $b_n$ is the constant term of the $n$th power of
\begin{equation}\label{eq:meldiag2}
\Lambda = \frac{(1+x_1)(1+x_2)(1+x_1+x_2)}{x_1 x_2}.
\end{equation}
Straub \cite[Rem.~1.4]{straub2014} observed that $a_n$ is the constant term of the $n$th power of
\begin{equation}\label{eq:straubdiag3}
\Lambda = \frac{(x_1+x_2)(x_3+1)(x_1+x_2+x_3)(x_2+x_3+1)}{x_1x_2x_3}.
\end{equation}
A short computation shows that the Newton polytopes of \eqref{eq:meldiag2} and \eqref{eq:straubdiag3} satisfy the conditions of Theorem~\ref{thm:samol}, and so $a_n$ and $b_n$ satisfy the D3 congruences. 
\subsection{D3 congruences for sporadic sequences}
Malik and Straub \cite[Thm.~3.1]{malik2016} proved that all 15 Ap\'{e}ry-like sequences satisfy the Lucas congruences. For all sequences except for $(\eta)$ and $s_{18}$ they did this by extending and applying a general method developed by McIntosh \cite{mcintosh1992}, while for $(\eta)$ and $s_{18}$ a much more delicate analysis was needed. From Theorems~\ref{thm:samol} and \ref{thm:newton} we immediately obtain
\begin{cor}
All the sporadic sequences, except possibly $(\eta)$, satisfy the D3 congruences.
\end{cor}
\subsection{\texorpdfstring{$p$}{p}-adic valuation of sporadic sequences}
Given a sequence of integers $\mathbf{u}=\{u_n\}_{n \ge 0}$ and a prime $p$, let
$\mathcal{Z}_p(\mathbf{u})=\{ 0 \le i \le p-1: p \mid u_i\}$. For any non-negative integer $n$ whose base-$p$ expansion is $n=\sum_{i=0}^{N} n_i p^i$, let $\alpha_p(\mathbf{u},n)=\#\{ 0\le i \le N: n_i \in \mathcal{Z}_p(\mathbf{u})\}$.

Let $v_p(n)$ be the $p$-adic valuation of an integer $n$. Recently, Delaygue proved that $v_p(u_n) \ge \alpha_p(\mathbf{u},n)$ under two conditions on $\{u_n\}_n$ \cite[Thm.~2]{delaygue2018}. The first condition is that $u_n = \CT(\Lambda^n)$ for some Laurent polynomial $\Lambda$ with integer coefficients, whose Newton polytope contains the origin as its only interior integral point. This holds for all the sporadic sequences except possibly $(\eta)$. The second condition involves a restriction on the differential operator annihilating $\sum_{n\ge 0} u_n z^n$. According to \eqref{eq:diff2}, \eqref{eq:diff3} and \cite[p.~23]{malik2016}, this condition holds for all 15 sequences. We thus obtain
\begin{cor}
	Every sporadic sequence $\{u_n\}_n$, except possibly $(\eta)$, satisfies $v_p(u_n) \ge \alpha_p(\mathbf{u},n)$ for every $n \ge 1$ and every prime $p$.
\end{cor}
\section{Constant term and binomial representations for sporadic sequences}\label{sec:constant}
\subsection{Review of diagonals, constant terms and binomial representations}
Given a power series $A(x_1,\ldots,x_d)=\sum_{(a_1,\ldots,a_d) \in \ZZ_{\ge 0}^d} c_{a_1,a_2,\ldots,a_d} x_1^{a_1} x_2^{a_2} \cdots x_d^{a_d} \in \CC[[x_1,\ldots,x_d]]$, we define its diagonal as the power series
\begin{equation}
\Delta_d(A) := \sum_{i \ge 0} c_{i,i,\ldots,i} u^i \in \CC[[u]].
\end{equation} 
Given a rational function $A(x_1,\ldots,x_d) \in \CC(x_1,\ldots,x_d)$ whose denominator does not vanish at the origin, we can identify it with its Taylor expansion around $\vec{0}$. For such $A$ we have the following result of Lipshitz.
\begin{thm}\label{thm:lip}\cite{lipshitz1988}
Let $A \in K(x_1,\ldots,x_d)$ ($K$ a field of characteristic $0$). If the denominator of $A$ does not vanish at the origin, then the diagonal $\Delta_d(A)$ satisfies a homogeneous linear differential equation with coefficients in $K[u]$.
\end{thm}
Given a Laurent series $f(x_1,\ldots,x_d) = \sum_{(a_1,\ldots,a_d) \in \ZZ^d} c_{a_1,a_2,\ldots,a_d} x_1^{a_1} x_2^{a_2} \cdots x_d^{a_d} \in \CC[[x_1,\ldots,x_d,x_1^{-1},\ldots,x_d^{-1}]]$, we define its constant term series as the power series
\begin{equation}
\CTS(f) := \sum_{i \ge 0} \CT(f^i) u^i \in \CC[[u]].
\end{equation}
The constant term series of $f$ is also known as its \emph{fundamental period}.
Suppose a Laurent polynomial $A(\mathbf{x})$ in $d$ variables becomes a polynomial when multiplied by $\prod_{i=1}^{d} x_i$. Then it is known that $A$ can be expressed as the diagonal of a rational function in $d+1$ variables, namely
\begin{equation}\label{eq:cts to diag}
\CTS( A)=\Delta_{d+1}\left( \frac{1}{1-x_1 x_2 \ldots x_{d+1} A(x_1,\ldots, x_d)} \right).
\end{equation}
In general, the diagonal of a rational function cannot be expressed as the constant term series of a Laurent polynomial. However, we have the following partial converse. This converse is a special case of the MacMahon Master Theorem \cite{macmahon1960}.
\begin{lem}\label{cor:diagtocts}
Given a square matrix $M=(M_{i,j})_{1 \le i,j \le n} \in \mathrm{Mat}_d(\CC)$, consider the rational function
\begin{equation}
A(x_1,\ldots,x_d) := \frac{1}{\det(I_d - M\cdot \mathrm{Diag}(x_1,\ldots,x_d))}.
\end{equation}
We have
\begin{equation}\label{eq:Delta as CT}
\Delta_{d}(A) = \CTS\left( \frac{ \prod_{i=1}^{d} (\sum_{j=1}^{d} M_{i,j} x_j)}{x_1 x_2 \cdots x_d} \right).
\end{equation}
As the rational function in the right-hand side of \eqref{eq:Delta as CT} is homogeneous of degree $0$, we may substitute $x_i=1$ for some $i$ without changing the constant term series.
\end{lem}
There are various benefits that come from representing a sequence as the diagonal of a rational function, some of which are described in \cite[p.~1987]{straub2014}. For instance, it allows one to extract the asymptotics of the sequence directly from the rational function.
 A much simpler benefit, which we make use of in Proposition~\ref{prop:formulas}, is that it allows us to easily extract binomial sum representations for the sequence\footnote{An example of a binomial sum is \eqref{eq:binom sums}. The interested reader may find a rigorous definition of `binomial sums' in \cite{bostan2017}.}.
\subsection{Known binomial sum representations}
In \cite[\S7]{zagier2009}, Zagier explains how the modular parametrization of a sequence can be used to obtain binomial representations for it. In practice, many of the binomial representations are found using educated guesses or computer searches. 

Binomial representations for $\mathbf{A}$ and $(\gamma)$ can be found in \cite{strehl1994, schmidt1995}, for $(\delta)$, $(\alpha)$ and $(\gamma)$ can be found in \cite{chan20102} and for $(\alpha)$ and $(\gamma)$ can be found in \cite{chan2011}. Additional representations for $\mathbf{A}$, $\mathbf{D}$, $\mathbf{E}$ and $(\alpha)$ are found in their OEIS pages \cite{oeis}, and we single out the following representation for $\mathbf{E}$,
\begin{equation}\label{eq:new e}
u_n = \sum_{k=0}^{n} \binom{n}{k}\binom{2k}{k} \binom{2n-2k}{n-k}.
\end{equation}
This formula for $\mathbf{E}$ was used in \cite{osburn2013} in establishing Gauss congruences of order 2; the original representation of $\mathbf{E}$, given by Zagier, involves powers of $4$, which get in the way of establishing supercongruences. A proof of \eqref{eq:new e} is sketched in \cite[p.~257]{delaygue2018}. 
\subsection{Known diagonal representations}\label{sec:rational examples}
Bostan, Lairez and Salvy \cite[Thm.~3.5]{bostan2017} proved (algorithmically) that if $\{u_n\}_{n \ge 0}$ is a binomial sum then the g.f. of $u_n$ is a diagonal of a rational function. Hence, every sporadic sequence is the diagonal of some rational function. We are aware of the following explicit diagonal representations for the 10 sequences $\mathbf{A}$, $\mathbf{B}$, $\mathbf{C}$, $\mathbf{D}$, $\mathbf{F}$, $(\gamma)$, $(\delta)$, $(\alpha)$, $(\epsilon)$, $s_7$ and $s_{10}$:
\begin{itemize}
	\item Zagier \cite{zagier2009} (cf. Verrill \cite{verrill1999}) already explained how his Sporadic sequences have constant term representations, but this has been worked out only in a few cases: $\mathbf{A}$ is the constant term sequence of 
\begin{equation}\label{eq:zagahomog}
	\frac{(x_1+x_2)(x_2+x_3)(x_3+x_1)}{x_1 x_2 x_3}
\end{equation}
and so of $(x_1+x_2)(x_2+1)(x_1+1)/x_1x_2$ as well, $\mathbf{B}$ is the constant term sequence of 
	\begin{equation}\label{eq:zagb}
	  \frac{x_1^3+x_2^3+x_3^3-3x_1x_2 x_3}{-x_1 x_2 x_3}
	\end{equation}
	and so of $(x_1^3+x_2^3+1-3x_1 x_2)/(-x_1 x_2)$ as well, and $\mathbf{F}$ is the constant term sequence of 
	\begin{equation}
		\frac{x_1^2 x_2+x_2^2 x_1+x_2^2 x_3 + x_3^2 x_2 + x_3^2  x_1 + x_1^2 x_3 -6x_1 x_2 x_3}{-x_1 x_2 x_3}
	\end{equation}
	and so of $(x_1^2 x_2+x_2^2 x_1+x_1^2 + x_2^2 + x_1+x_2-6x_1 x_2)/(-x_1 x_2)$ as well.
	\item Various rational functions whose diagonal is the g.f. for $(\gamma)$ are given in \cite{christol1984, christol1985, bostan2013,  straub2014}. Of particular interest is the one found by Straub, which involves 4 variables, while others involve between 5 to 8 variables. His rational function is \cite[Thm.~1.1]{straub2014}
\begin{equation}\label{eq:straub gamma}
((1-x_1-x_2)(1-x_3-x_4)-x_1x_2x_3x_4)^{-1}.
\end{equation}
Recently, Gheorghe Coserea added to the OEIS \cite{oeis} two rational functions, in 4 variables, whose diagonal is the g.f. of $(\gamma)$: $(1 - (x_1 x_2 x_3 x_4 + x_1 x_2 x_4 + x_3 x_4 + x_1 x_2 + x_1 x_3 + x_2 + x_3))^{-1}$ and $(1 - x_4(1+x_1)(1+x_2)(1+x_3)(x_1 x_2 x_3 + x_2 x_3 + x_2+x_3+1))^{-1}$.
\item Straub observed \cite[Ex.~3.4]{straub2014} that the diagonal of
\begin{equation}\label{eq:straub d}
((1-x_1-x_2)(1-x_3)-x_1x_2x_3)^{-1}
\end{equation}
is the g.f. of $\mathbf{D}$. Bostan, Boukraa, Maillard and Weil \cite{bostan2015}, as well as Coserea \cite{oeis}, have found very similar representations. 
\item Straub observed that
\begin{equation}\label{eq:straub franel}
((1-x_1)(1-x_2)(1-x_3)-x_1x_2x_3)^{-1}, \quad (1 - x_1 - x_2 - x_3 + 4x_1 x_2 x_3)^{-1}
\end{equation}
are both diagonals for $\mathbf{A}$ \cite[Ex.~3.6,~3.7]{straub2014}. In the OEIS, Coserea provides three more rational functions, in 3 variables as well, whose diagonal is the g.f. of $\mathbf{A}$: 
\begin{equation}\label{eq:coserea franel}
(1 - x_1 x_2 - x_2 x_3- x_1 x_3 - 2x_1 x_2 x_3)^{-1}, \quad
(1 + x_1 + x_2 + x_3 + 2(x_1 x_2 + x_2 x_3 + x_3 x_1) + 4x_1 x_2 x_3)^{-1}
\end{equation}
and $(1 - (x_1+ x_2 + x_3) + x_1 x_2 + x_2 x_3 + x_3 x_1  - 2x_1 x_2 x_3)^{-1}$.
\item Straub also observed \cite[Eq.~(32)]{straub2014} that  the g.f. for $(\delta)$ coincides with the diagonal of 
\begin{equation}\label{eq:straub delta}
(1 - (x_1 + x_2 + x_3 - x_4) -27 x_1x_2x_3x_4)^{-1}.
\end{equation}
\item The diagonal of $
((1-x_1)(1-x_2)(1-x_3)(1-x_4)-x_1 x_2 x_3 x_4)^{-1}$ is the g.f. for $s_{10}$, see Straub \cite[Ex.~3.6]{straub2014}. Coserea \cite{oeis} also provides the rational function $
(1 - (x_2 x_3 x_4 + x_1 x_3 x_4 + x_1 x_2 x_4  + x_1 x_2 x_3 + x_1 x_4 + x_2 x_3))^{-1}$.
\item Sequences $\mathbf{C}$ and $(\alpha)$ are the constant term sequences of the Laurent polynomial $P_d(x_1,\ldots,x_d) = \left( x_1+x_2+\ldots+x_d \right) \left(\frac{1}{x_1}+\ldots + \frac{1}{x_d}\right)$ with $d=3$ and $d=4$, respectively (see \cite[Ex.~6]{borwein2013}, \cite[Ex.~3.17]{borwein2016}, \cite[Eq.~(8)]{borwein2011}). By \eqref{eq:cts to diag}, this also gives diagonal representations for $\mathbf{C}$ and $(\alpha)$.
\item Coserea lists the functions $(1 - (x_1^2x_2 + x_2^2x_3 - x_3^2x_1 + 3x_1 x_2 x_3))^{-1}$, $(1 - (x_1^3 + x_2^3 - x_3^3 + 3x_1 x_2 x_3))^{-1}$ and $(1 + x_1^3 + x_2^3 + x_3^3 - 3x_1 x_2 x_3)^{-1}$ as rational functions whose diagonal is the g.f. for $\mathbf{B}$ \cite{oeis}.
\item Coserea has also found that the diagonal of $(1 - (x_4(x_1x_2+x_2x_3+x_3x_1)+ x_1 x_2 + x_1 x_3 + x_2 +x_3))^{-1}$ is the g.f. for $s_{7}$.
\item In \cite[Ex.~4.2]{straub2015}, Straub and Zudilin prove that the g.f. of $(\epsilon)$ is the diagonal of the Kauers–Zeilberger rational function $(1-(x_1+x_2+x_3+x_4)+2e_3(x_1,x_2,x_3,x_4)+4x_1x_2x_3x_4)^{-1}$, where $e_3$ is the third elementary symmetric polynomial.
\end{itemize}
In the proof of \cite[Thm.~1.1]{straub2014}, Straub observes that the rational function \eqref{eq:straub gamma} may be written as $\det(I_4 - M\cdot \mathrm{Diag}(x_1,\ldots,x_4))^{-1}$ where
\begin{equation}
M = \begin{pmatrix} 1  & 1 & 1 & 0 \\ 1 & 1 & 0 & 0 \\ 0 & 0 & 1 & 1 \\ 0 & 1 & 1 & 1  \end{pmatrix},
\end{equation}
and so by Lemma~\ref{cor:diagtocts} it gives rise to a constant term representation for $(\gamma)$, the one defined by \eqref{eq:straubdiag3}. Representations of the form $\det(I_d - M \cdot \mathrm{Diag}(x_1,\ldots,x_d))^{-1}$ are hidden in some of the other functions that we listed, as we now show. The rational function in \eqref{eq:straub d} can be written as $\det(I_3 - M \cdot \mathrm{Diag}(x_1,x_2,x_3))^{-1}$ where 
\begin{equation}
M= \begin{pmatrix} 1  & 1 & 0  \\ 1 & 1 & 1  \\ 1 & 0 & 1\end{pmatrix},
\end{equation}
and by Lemma~\ref{cor:diagtocts}, we find that the constant term series of \eqref{eq:meldiag2} is the g.f. of $\mathbf{D}$. The rational functions in \eqref{eq:straub franel} and \eqref{eq:coserea franel} can be written as $\det(I_3- M \cdot \mathrm{Diag}(x_1,x_2,x_3))^{-1}$ for
\begin{equation}
M=\begin{pmatrix} 1  & 1 & 0  \\ 0 & 1 & 1  \\ 1 & 0 & 1\end{pmatrix} ,\begin{pmatrix} 1  & 1 & 1  \\ 1 & 1 & -1  \\ 1 & -1 & 1\end{pmatrix},\begin{pmatrix} 0  & 1 & 1  \\ 1 & 0 & 1  \\ 1 & 1 & 0\end{pmatrix}, \begin{pmatrix} -1  & 1 & 1  \\ -1 & -1 & 1  \\ -1 & -1 & -1\end{pmatrix}.
\end{equation}
\subsection{New constant term representations}\label{sec:newconstant}
The discussion in \S\ref{sec:rational examples} shows that there are known constant term representations for $\mathbf{A}$, $\mathbf{B}$, $\mathbf{C}$, $\mathbf{F}$, $(\gamma)$ and $(\alpha)$, and also that a previously known diagonal representation for $\mathbf{D}$ gives rise to constant term representation through Lemma~\ref{cor:diagtocts}.

We call a Laurent polynomial \emph{good} if it may be factored into a product of Laurent polynomial, each of which has coefficients in the set $\{-1,0,1\}$.
We have conducted a computer search of the following form: we enumerated good Laurent polynomials $f$ (in several variables) with bounded degree and coefficients. For each such polynomial we tested whether $\CTS(f)$ is the g.f. of a sporadic sequence. For each of the 15 sequences, the search yielded more than one Laurent polynomial whose $\CTS$ is its g.f.. The reason for looking only at good Laurent polynomials is that they lead to binomial sum representations which do not involve powers of integers other than $\pm 1$ (as in Proposition~\ref{prop:formulas}).

We list some of the Laurent polynomials we have found in the following theorem.
\begin{proposition}\label{prop:rep}
The constant term series of the Laurent polynomials in the second column are g.f.s of the sequences in the first column:	
\begin{center}
	\begin{tabular}{l|p{110mm}}
		\hline
		Sequence & Laurent polynomial  \\[0.5ex]
		\hline
		$\mathbf{A}$ & $(xy)^{-1} (x+1)(y+1)(x+y)$ \newline $(xy)^{-1} (x + y + 1) (x y + y + 1)$\\
		$\mathbf{B}$ & $(-xy)^{-1} (x + y + 1) (x^2 +y^2- xy - x - y + 1)$  \newline 
		$(-xy)^{-1} (x^2 - x y + y^2 - x - y) (x^2 - x y + y^2 + x + y + 1) $ \\
		$\mathbf{C}$ & $(xy)^{-1} (x + y + 1) (xy + x + y)$  \newline $(-xxy)^{-1} (x + y - 1) (x + y + 1) (xy - 1)$\\
		 $\mathbf{D}$ & $(xy)^{-1}(x+1)(y+1)(x+y+1)$ \newline $(xy)^{-1} (y - 1) (x - 1) (x - y + 1) (-x + y + 1)$ \\
		$\mathbf{E}$ & $(-xy)^{-1}(xy + x + y - 1)  (xy - x - y - 1)$ \\
		$\mathbf{F}$ & $(xy)^{-2}(-x+y+1)(x-y+1)(x+y-1)(x+y+1)(x^2+y^2+1)$ \\
	\end{tabular}
\end{center}

\begin{center}
	\begin{tabular}{l|p{110mm}}
		\hline
Sequence & Laurent polynomial  \\[0.5ex]
		\hline	
		$(\delta)$ & $(xyz)^{-1}(y - z + 1) (-x + y + z)  (x + z + 1) (x + y - 1)$ \newline $(xyz)^{-1}(xy + yz+zx)(x^2 +y^2+z^2 - xy-yz-zx+ x + y + z + 1)$ \\
		$(\eta)$ & $(xyz)^{-1} (zx + xy - yz - x - 1) (xy + yz  -zx - y - 1) ( yz + zx -xy- z - 1)  $\\
		$(\alpha)$ & $(xyz)^{-1}(-x - y - z + 1) (x - y) (x - y + z + 1)(x + y - z + 1)$ \newline $(xyz)^{-1}(x + y + z + 1) (xyz + xy + yz + zx)$  \\
		$(\epsilon)$ & $(xyz)^{-1}(x+1)(y+1)(z+1)(x+y+z+1)$  \\
		 $(\zeta)$ & $(xyz)^{-1}(x + y + z) (x + y + z  + xy+yz+zx + xyz)$\\
		$(\gamma)$ & $(xyz)^{-1} (y + z) (x + 1) (x + y + 1)  (x + y + z)$\newline $(xyz)^{-1} (xy+yz+zx+x+z)(xy+yz+zx+x+y+z+1)$
\end{tabular}
\end{center}

\begin{center}
\begin{tabular}{l|p{110mm}}
\hline
Sequence & Laurent polynomial  \\[0.5ex]
\hline
		$s_7$ & $(xyz)^{-1}(x-1)(y + 1) (x + z)  (y-x-z+1)$  \\
		$s_{10}$ &  $(xyz)^{-1} (x + 1) (y + 1) (z + 1) (xyz + 1)$  \\
		$s_{18}$ &  $(-xyz)^{-1}(x^2 +y^2+z^2 - xy - yz-zx - x + y - z) (x^2 +y^2+z^2+ xy +yz -zx+ x + y + z)	$	
	\end{tabular}
\end{center}
\end{proposition}
As can be seen, we were always able to find good \emph{symmetric} polynomials, except in the cases of $(\gamma)$, $s_7$ and $s_{18}$. For $(\gamma)$, $P = (xyz)^{-1}(x + y + z + 1) (x^2y + xy^2 + y^2 z+y z^2 + z^2 x + zx^2  + xy + yz + zx + 2xyz)$ is a symmetric polynomial whose constant term series is (empirically) the g.f. of $(\gamma)$, but $P$ is not good according to our definition. For $s_{18}$, $P = (xyz)^{-1}(xy + yz +zx + x + y + z) (x^2 +y^2 + z^2 - xy -yz-zx-x-y-z+1)$ is a symmetric polynomial whose constant term series is (empirically) the g.f. of $s_{18}$, and is good according to our definition, but we do not know how to prove it. Note that the first polynomial for $\mathbf{B}$ was given in \cite{verrill1999}, although it was not factored; its factorization is useful in the proofs of Proposition~\ref{prop:formulas} and Theorem~\ref{thm:super}.
\begin{proof}[Proof of Proposition~\ref{prop:rep}]
For each of the Laurent polynomials in the proposition, we may easily compute the first three terms and see that they agree with the first three terms of the respective sequence. It now suffices to compute the differential operator annihilating the constant term series of the polynomial, and that it coincides with \eqref{eq:diff2} or \eqref{eq:diff3}.
 
The proof of Theorem~\ref{thm:lip} is constructive and may give rise to such a differential operator (by appealing to \eqref{eq:cts to diag}). However, the algorithm arising from Lipshitz's theorem is often not efficient enough for practical purposes. However, Zeilberger's Creative Telescoping algorithm \cite{zeilberger1991}, as generalized by Chyzak \cite{chyzak2000}, gives a much faster way to compute the operator. See \cite{bostan2011} for more details.

The algorithm is implemented in the package $\textit{Mgfun}$.\footnote{Version 4.1, available from \url{https://specfun.inria.fr/chyzak/mgfun.html}.} To find an operator annihilating the diagonal of $f(s,t,x)$, we run the command
\begin{verbatim}
P_, S_, T_ := op(op(Mgfun:-creative_telescoping(F, x::diff, [s::diff, t::diff])))
\end{verbatim}
on $F:=(st)^{-1}f(s,t/s,x/t)$, and $P\_$ will be the operator. So if we want to find the operator annihilating $\CTS(f(s,t))$, we shall run it on $F:=(st)^{-1}(1-xf(s,t/s))^{-1}$. To find an operator annihilating the diagonal of $f(s,t,x,w)$, we run the command
\begin{verbatim}
P_, S_, T_, X_ := op(op(Mgfun:-creative_telescoping(F, w::diff, [s::diff, t::diff, x::diff])))
\end{verbatim}
on $F:=(stx)^{-1}f(s,t/s,x/t,w/x)$, and $P\_$ will be the operator. So if we want to find the operator annihilating $\CTS(f(s,t,x))$, we shall run it on $F:=(stx)^{-1}(1-wf(s,t/s,x/t))^{-1}$.

In the two-variate case, the algorithm returns the expected operator in under a minute on an average laptop. For the three-variate case, the algorithm takes a few minutes when we ran it on the polynomials for $(\epsilon)$, $(\zeta)$, $(\gamma)$, $s_7$ and $s_{10}$. For $(\delta)$, $(\eta)$, $(\alpha)$ and $s_{18}$ we had to `massage' the Laurent polynomials slightly before running the program, as we now explain.

For $(\delta)$, we consider $f=(xyz)^{-1}(y - z + 1) (-x + y + z)  (x + z + 1) (x + y - 1)$. We know by Lemma~\ref{cor:diagtocts} that for the matrix 
\[
M = \begin{pmatrix} 0 & 1 & -1 & 1 \\ -1 & 1 & 1 & 0 \\ 1 & 0 & 1 & 1 \\ 1 & 1 & 0 &-1  \end{pmatrix},
\]
we have
\begin{equation}
\CTS\left( f \right) = \Delta_{4}\left(\frac{1}{\det(I_4 - M\cdot \mathrm{Diag}(x_1,\ldots,x_4))}\right).
\end{equation}
Instead of running the algorithm on $(stx)^{-1}(1-w f(s,t/s,x/t))^{-1}$, we ran it on $(stx)^{-1} g(s,t/s,x/t,w/x)$ where $g:=\det(I_4 - M\cdot \mathrm{Diag}(x_1,\ldots,x_4))$, and it proved to be much quicker. The same trick works for the first polynomial of $(\alpha)$. The second polynomial for $\alpha$ comes from the polynomial $\left( x_1+x_2+x_3+x_4 \right)\left(1/x_1 + \ldots + 1/x_4\right)$ mentioned in \S\ref{sec:rational examples}, in which we plugged $x_4=1$.

For the second polynomial $f$ for $(\delta)$, note that $\CT(f^n) = \CT(g^n)$ for $g=f(\frac{1}{xyz},\frac{1}{yz},\frac{1}{z})$. We ran the algorithm on $(stx)^{-1}(1-w g(s,t/s,x/t))^{-1}$. For $s_{18}$, the same trick works.

Finally, for $(\eta)$ we use a different implementation of Creative Telescoping, which turned out to be faster for this example, namely the \textit{Mathematica} package \textit{HolonomicFunctions} \footnote{Version 1.7.3, available from \url{https://www3.risc.jku.at/research/combinat/software/ergosum/installation.html}.} by Christoph Koutschan. To find the recurrence satisfied by $\CT(f^n)$ for the polynomial of $(\eta)$, we ran
\begin{verbatim}
CreativeTelescoping[((z x + x y - y z - x - 1) (x y + y z - z x - y - 1) 
  (y z + z x - x y - z - 1)/(x y z))^n/(x y z), Der[x], {Der[y], Der[z], S[n]}]
CreativeTelescoping[First[%], Der[y], {Der[z], S[n]}]
CreativeTelescoping[First[%], Der[z], {S[n]}]
\end{verbatim}
which outputs the expected recurrence.
\end{proof}

\subsection{Proof of Proposition~\ref{prop:formulas}}
 We use bold Latin letters $\mathbf{a},\mathbf{b},\ldots$ as short for sequences $(a_1,a_2,a_3),(b_1,b_2,b_3),\ldots$. If e.g. $\mathbf{a}$ sums to $n$  we write $\binom{n}{\mathbf{a}}$ for the multinomial coefficient $\binom{n}{a_1,a_2,a_3}$. For the first part, note that
	\begin{equation}
	\begin{split}
	\left(\frac{(x+y+1)(x^2+y^2-xy-x-y+1)}{-xy}\right)^n &= \frac{(x+y+1)^n(x+y\omega + \omega^2)^n(x+y\omega^2 + \omega)^n}{(-xy)^n} \\
	&= (-xy)^{-n}\sum_{a_1+a_2+a_3 = n} x^{a_1} y^{a_2} \binom{n}{\mathbf{a}} \sum_{b_1+b_2+b_3 = n} x^{b_1} y^{b_2}  \omega^{b_2+2b_3} \binom{n}{\mathbf{b}} \\
	&\qquad   \sum_{c_1+c_2+c_3 = n} x^{c_1} y^{c_2} \omega^{2c_2+c_3} \binom{n}{\mathbf{c}},
	\end{split}
	\end{equation}
	where $\omega=e^{2\pi i /3}$. By Proposition~\ref{prop:rep}, $\CTS((x+y+1)(x^2+y^2-xy-x-y+1)/(-xy))$ is the generating series for $B_n$, and so
	\begin{equation}\label{eq:bunity}
	B_n = (-1)^n\sum_{\substack{a_1,a_2,a_3\\b_1,b_2,b_3\\c_1,c_2,c_3}} \binom{n}{\mathbf{a}}\binom{n}{\mathbf{b}}\binom{n}{\mathbf{c}} \omega^{b_2+2b_3+2c_2+c_3},
	\end{equation}
	where the sum is over tuples $(a_1,a_2,a_3,b_1,b_2,b_3,c_1,c_2,c_3)$ of non-negative integers satisfying $a_1+a_2+a_3 = b_1+b_2+b_3 = c_1+c_2+c_3 = n$ and  $a_i + b_i+c_i  =n$ for $i=1,2,3$. Observe that $\Re(2\omega^k) = 3\cdot 1_{3 \mid k} - 1$. Since $B_n$ is real, we may take the real part of \eqref{eq:bunity}, obtaining
	\begin{equation}\label{eq:btwosums}
	2(-1)^n B_n = 3\sum_{\substack{a_1,a_2,a_3\\b_1,b_2,b_3\\c_1,c_3,c_3}} \binom{n}{\mathbf{a}}\binom{n}{\mathbf{b}} \binom{n}{\mathbf{c}} -\CT\left( \left(\frac{(x+y+1)^3}{xy}\right)^n\right),
	\end{equation}
	where we sum over tuples $(a_1,a_2,a_3,b_1,b_2,b_3,c_1,c_2,c_3)$ in $S(n)$. The second term evaluates to $\binom{3n}{n,n,n}$ by the multinomial theorem, as needed.	For the second part, note that
	\begin{equation}
	\begin{split}
	&\left( \frac{ (-x + y + 1)  (x - y + 1)  (x + y - 1)  (x + y + 1)  (x^2 + y^2 +1)}{x^2y^2} \right)^n \\
	& \qquad \qquad  = \frac{(-x+y+1)^n (x-y+1)^n (x+y-1)^n ( x+y+1)^n(x^2+y^2+1)^n}{(xy)^{2n}}\\
	&\qquad \qquad = (xy)^{-2n} \sum_{a_1+a_2+a_3 = n} x^{a_1} y^{a_2}  (-1)^{a_1} \binom{n}{\mathbf{a}} \sum_{b_1+b_2+b_3=n} x^{b_1} y^{b_2} (-1)^{b_2}\binom{n}{\mathbf{b}}\\
	& \qquad \qquad \qquad \qquad \qquad\qquad \sum_{c_1+c_2+c_3 = n} x^{c_1} y^{c_2}  (-1)^{c_3} \binom{n}{\mathbf{c}} \sum_{d_1+d_2+d_3=n} x^{d_1} y^{d_2} \binom{n}{\mathbf{d}}\\
	& \qquad \qquad \qquad \qquad \qquad\qquad \sum_{e_1+e_2+e_3 = n} x^{2e_1} y^{2e_2} \binom{n}{\mathbf{e}}.
	\end{split}
	\end{equation}
	Since $\CTS( (-x + y + 1)  (x - y + 1)  (x + y - 1)  (x + y + 1)  (x^2 + y^2 + 1)/(x^2y^2))$ is the generating series for $F_n$ (by Proposition~\ref{prop:rep}), one obtains \eqref{eq:fn form}.	For the third part, note that
	\begin{equation}
	\begin{split}
	&\left(  \frac{(y - z + 1)  (-x + y + z)  (x + z + 1)  (x + y - 1)}{xyz} \right)^n \\
	\qquad  \qquad \qquad &= \frac{(y-z+1)^n (-x+y+z)^n ( x+z+1)^n (x+y-1)^n}{(xyz)^n} \\
	\qquad \qquad \qquad  &= (xyz)^{-n} \sum_{a_1+a_2+a_3=n} y^{a_1} z^{a_2} (-1)^{a_2} \binom{n}{\mathbf{a}} \sum_{b_1+b_2+b_3=n} x^{b_1} y^{b_2} z^{b_3} (-1)^{b_1} \binom{n}{\mathbf{b}} \\
	&  \qquad \qquad \qquad \qquad\qquad \sum_{c_1+c_2+c_3=n} x^{c_1} z^{c_2} \binom{n}{\mathbf{c}}  \sum_{d_1+d_2+d_3=n} x^{d_1} y^{d_2} (-1)^{d_3} \binom{n}{\mathbf{d}} .
	\end{split}
	\end{equation}
	Since $\CTS( (y-z+1)  (-x+y + z)  (x + z + 1)  (x + y - 1)/(xyz))$ is the generating series for $\AZ_n$ (by Proposition~\ref{prop:rep}), one obtains \eqref{eq:azn form}. \qed

\subsection{Proof of Theorem~\ref{thm:newton}} For $\mathbf{A}\mhyphen\mathbf{E}$, we can check that the polynomials in Proposition~\ref{prop:rep} satisfy all the conditions of the theorem by drawing the corresponding Newton polygons (in case there is more than one polynomial, we check the first one only). For instance, 

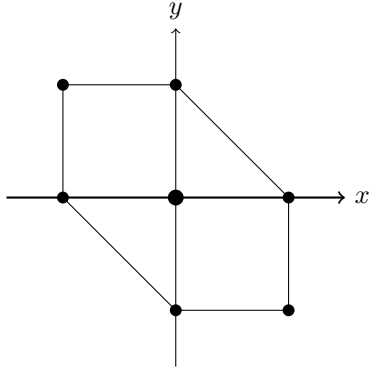
\begin{center}
\begin{tikzpicture}[scale=1.5]
	\draw [->,thick] (-1.5,0)--(1.5,0) node[right]{$x$};
	\draw [->] (0,-1.5)--(0,1.5) node[above]{$y$};
	\draw (1,0) -- (0,1) ;
	\draw (0,1) -- (-1,1) ;
	\draw (-1,1) -- (-1,0) ;
	\draw (-1,0) -- (0,-1) ;
	\draw (0,-1) -- (1,-1) ;
	\draw (1,-1) -- (1,0) ;
	\fill (1,0) circle (1.5pt);
	\fill (0,1) circle (1.5pt);
	\fill (-1,1) circle (1.5pt);
	\fill (-1,0) circle (1.5pt);
	\fill (0,-1) circle (1.5pt);
	\fill (1,-1) circle (1.5pt);
	\fill (0,0) circle (2pt);
	\end{tikzpicture}
\end{center}

is the Newton polygon of the polynomial corresponding to $\mathbf{A}$ in Proposition~\ref{prop:rep}. For $\mathbf{F}$, the polynomial $P(x,y) = (xy)^{-2}(-x+y+1)(x-y+1)(x+y-1)(x+y+1)(x^2+y^2+1)$ satisfies the first condition of the theorem, but not the later two. It turns out that $P(x,y) = Q(x^2,y^2)$ for $Q(x,y) = (-xy)^{-1}(x+y+1)(x^2+y^2-2xy-2x-2y+1)$. Since $\CT(P^n) = \CT(Q^n)$, we see that $Q(x,y)$ satisfies the second condition. Finally, we check the third condition, regarding the Newton polygon of $Q$, which is given below:

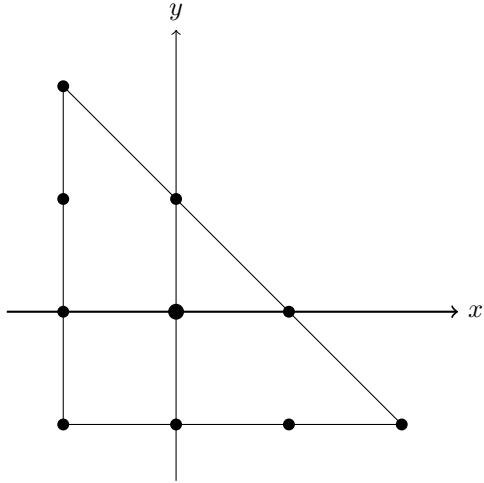
\begin{center}
	\begin{tikzpicture}[scale=1.5]
	\draw [->,thick] (-1.5,0)--(2.5,0) node[right]{$x$};
	\draw [->] (0,-1.5)--(0,2.5) node[above]{$y$};
	\draw (2,-1) -- (-1,2) ;
	\draw (-1,2) -- (-1,-1) ;
	\draw (-1,-1) -- (2,-1) ;
	\fill (2,-1) circle (1.5pt);
	\fill (-1,2) circle (1.5pt);
	\fill (1,0) circle (1.5pt);
	\fill (0,1) circle (1.5pt);
	\fill (1,-1) circle (1.5pt);
	\fill (-1,1) circle (1.5pt);
	\fill (0,-1) circle (1.5pt);
	\fill (-1,0) circle (1.5pt);
	\fill (-1,-1) circle (1.5pt);
	\fill (0,0) circle (2pt);
	\end{tikzpicture}
\end{center}

Indeed, $(0,0)$ is its only interior point. For $(\delta)\mhyphen(\gamma)$ and $s_7\mhyphen s_{18}$, drawing a picture is harder. In $\textit{SageMath}$, one can check whether a Laurent polynomial in 3 variables contains the origin as its only interior point using the following procedure:
\begin{verbatim}
def checkInterior(poly):
  polyhedron = Polyhedron(vertices = poly.exponents())
  for pt in polyhedron.integral_points():
    if tuple(pt) != (0,0,0) and polyhedron.interior_contains(pt):
     return False
  return polyhedron.interior_contains((0,0,0))
\end{verbatim}
For instance, the following code checks that the first polynomial of $(\delta)$ (from Proposition~\ref{prop:rep}) satisfies the conditions of the theorem:
\begin{verbatim}
R.<x,y,z> = LaurentPolynomialRing(QQ,3)
P = (y-z+1)*(-x+y+z)*(x+z+1)*(x+y-1) / (x*y*z)
checkInterior(P)
>> True
\end{verbatim}
This allowed us to verify that the Newton polytopes contain the origin as their only interior point, except in the case of $(\eta)$, where $(1,0,0)$ and $(1,1,0)$ (as well as their permutations) are additional interior points. \qed

\section{Proof of Theorem~\ref{thm:super}}
\subsection{Auxiliary results}
The following lemma is a classical result of Jacobsthal \cite{brun1949}, discovered independently by Kazandzidis \cite{kazandzidis1969} (cf. \cite[Ch.~7.1.6]{robert2000}).
\begin{lem}\label{lem:jacbosthal}
Let $p \ge 3$ be a prime and $n \ge m \ge 0$ be integers. Then
\begin{equation}
\frac{\binom{pn}{pm}}{\binom{n}{m}} \equiv 1 \bmod p^{v_p(nm(n-m))+3-\delta_{p,3}}
\end{equation}
where $\delta_{p,3}=1$ if $p=3$ and $\delta_{p,3}=0$ otherwise.
\end{lem}
The identity $\binom{n}{m} = \binom{n-1}{m-1}n/m$ yields the following congruence.
\begin{lem}\label{lem:lower}
Let $p$ be a prime and $n \ge m \ge 0$ be integers. Suppose that $v_p(n) \ge v_p(m)$. Then
\begin{equation}
\binom{n}{m} \equiv 0 \bmod p^{v_p(n)-v_p(m)}.
\end{equation}
\end{lem}
\subsection{Conclusion of proof}
As $\binom{3n}{n,n,n} = \binom{3n}{2n}\binom{2n}{n}$ satisfies the Gauss congruences of order $2$ by Lemma~\ref{lem:jacbosthal}, and so does $(-1)^n$, we deal with $\tilde{B}_n = 2(-1)^n B_n + \binom{3n}{n,n,n}$ instead of $B_n$.

The vector $(\mathbf{a},\mathbf{b},\mathbf{c})$ stands for $(a_1,a_2,\ldots,c_2,c_3) \in \ZZ^9$. Let $k,n \ge 1$ and let $p \ge 3$ be a prime. We write $p\cdot (\mathbf{a},\mathbf{b},\mathbf{c})$ for $(p\cdot \mathbf{a},p\cdot \mathbf{b},p\cdot \mathbf{c})$, and $p \mid (\mathbf{a},\mathbf{b},\mathbf{c})$ if all the entries of $(\mathbf{a},\mathbf{b},\mathbf{c})$  are divisible by $p$. Given $(\mathbf{a},\mathbf{b},\mathbf{c}) \in S(m)$ we write \begin{equation}
B(\mathbf{a},\mathbf{b},\mathbf{c}) = \binom{m}{\mathbf{a}}\binom{m}{\mathbf{b}} \binom{m}{\mathbf{c}}.
\end{equation}
Let $p$ be an odd prime and let $n, k \ge 1$. By Proposition~\ref{prop:formulas} we may write $\tilde{B}_{np^k}-\tilde{B}_{np^{k-1}}$ as
\begin{equation}\label{eq:diff s123}
\tilde{B}_{np^k}-\tilde{B}_{np^{k-1}} = 3(S_1 + S_2 + S_3)
\end{equation}
where
\begin{equation}
\begin{split}
S_1 &= \sum_{\substack{(\mathbf{a},\mathbf{b},\mathbf{c}) \in S(np^k)\\ p\nmid (\mathbf{a},\mathbf{b},\mathbf{c})}} B(\mathbf{a},\mathbf{b},\mathbf{c}),\\
S_2 &= \sum_{ (\mathbf{a},\mathbf{b},\mathbf{c}) \in S(np^{k-1}) } \left( B(p\cdot  (\mathbf{a},\mathbf{b},\mathbf{c})) - B( \mathbf{a},\mathbf{b},\mathbf{c}) \right), \\
S_3 &= \sum_{\substack{ p\cdot (\mathbf{a},\mathbf{b},\mathbf{c}) \in S(np^{k}) \\ (\mathbf{a},\mathbf{b},\mathbf{c}) \notin S(np^{k-1})  }} B(p\cdot  (\mathbf{a},\mathbf{b},\mathbf{c})) 
\end{split}
\end{equation}
Here we have used the following observation: if $(\mathbf{a},\mathbf{b},\mathbf{c}) \in S(np^{k-1})$ then $p\cdot (\mathbf{a},\mathbf{b},\mathbf{c}) \in S(np^{k})$.

Let $(\mathbf{a},\mathbf{b},\mathbf{c}) \in S(np^k)$ with $p\nmid (\mathbf{a},\mathbf{b},\mathbf{c})$. Suppose without loss of generalization that $p \nmid a_1$. At least one of $b_1$ and $c_1$ is indivisible by $p$ as well, because of the condition $a_1+b_1+c_1=np^k$. Suppose without loss of generalization that $p \nmid b_1$. The integer $B(\mathbf{a},\mathbf{b},\mathbf{c})$ is divisible by $\binom{np^k}{a_1}\binom{np^k}{b_1}$, which by Lemma~\ref{lem:lower} implies that $B(\mathbf{a},\mathbf{b},\mathbf{c}) \equiv 0 \bmod p^{2k}$, and so $S_1 \equiv 0 \bmod p^{2k}$.

We now deal with $S_2$. Let $(\mathbf{a},\mathbf{b},\mathbf{c}) \in S(np^{k-1})$. We may write both $B(p\cdot  (\mathbf{a},\mathbf{b},\mathbf{c}))$  and $B(\mathbf{a},\mathbf{b},\mathbf{c})$ as a product of $6$ binomial coefficients, namely
\begin{equation}
\begin{split}
B(p\cdot (\mathbf{a},\mathbf{b},\mathbf{c})) & = \binom{np^k}{pa_1} \binom{p(a_2+a_3)}{pa_2} \binom{np^k}{pb_1} \binom{p(b_2+b_3)}{pb_2} \binom{np^k}{pc_1} \binom{p(c_2+c_3)}{pc_2},\\
B (\mathbf{a},\mathbf{b},\mathbf{c}) & = \binom{np^{k-1}}{a_1} \binom{a_2+a_3}{a_2} \binom{np^{k-1}}{b_1} \binom{b_2+b_3}{b_2} \binom{np^{k-1}}{c_1} \binom{c_2+c_3}{c_2}.
\end{split}
\end{equation}
Applying Lemma~\ref{lem:jacbosthal} $6$ times, we obtain that
\begin{equation}\label{eq:b d div}
\frac{B(p\cdot  (\mathbf{a},\mathbf{b},\mathbf{c}))}{B(\mathbf{a},\mathbf{b},\mathbf{c})} \equiv 1 \bmod p^{d}
\end{equation}
where
\begin{equation}
\begin{split}
d = \min\{ &k+v_p(pa_1)+v_p(p(a_2+a_3)), v_p(p^3a_2a_3(a_2+a_3)),\\  &k+v_p(pb_1)+v_p(p(b_2+b_3)), v_p(p^3b_2b_3(b_2+b_3)),\\ &k+v_p(pc_1)+v_p(p(c_2+c_3)), v_p(p^3c_2c_3(c_2+c_3))\} - 1.
\end{split}
\end{equation}
Without loss of generalization, suppose that $d = k+v_p(a_1)+v_p(a_2+a_3)+1$ or $d =v_p(a_2a_3(a_2+a_3))+2$. In the first case, since $\binom{np^{k-1}}{a_1} \mid B(\mathbf{a},\mathbf{b},\mathbf{c})$, Lemma~\ref{lem:lower} implies that $B(\mathbf{a},\mathbf{b},\mathbf{c})\equiv 0 \bmod p^{\max\{k-1-v_p(a_1),0\}}$ and so 
\begin{equation}
B(p\cdot  (\mathbf{a},\mathbf{b},\mathbf{c})) \equiv B(\mathbf{a},\mathbf{b},\mathbf{c}) \bmod p^{e},
\end{equation}
where $e \ge  d + k-1 -v_p(a_1) \ge 2k + v_p(a_2+a_3) \ge 2k$, and so the summand of $S_2$ corresponding to $(\mathbf{a},\mathbf{b},\mathbf{c})$ vanishes modulo $p^{2k}$. Suppose now that we are in the second case, $d =v_p(a_2a_3(a_2+a_3))+2$. Since $a_i+b_i+c_i=p^{k-1}n$ for each $i\in \{1,2,3\}$, we have 
\begin{equation}
B(\mathbf{a},\mathbf{b},\mathbf{c}) = \prod_{i=1}^{3} \binom{np^{k-1}}{a_i,b_i,c_i},
\end{equation}
and in particular $\binom{np^{k-1}}{a_2}\binom{np^{k-1}}{a_3} \mid B(\mathbf{a},\mathbf{b},\mathbf{c})$, so that by Lemma~\ref{lem:lower} we have 
\begin{equation}\label{eq:b exp sum}
B(\mathbf{a},\mathbf{b},\mathbf{c}) \equiv 0 \bmod p^{\max\{k-1-v_p(a_2),0\} + \max\{k-1-v_p(a_3),0\}}.
\end{equation}
From \eqref{eq:b d div} and \eqref{eq:b exp sum} we obtain
\begin{equation}
B( p\cdot  (\mathbf{a},\mathbf{b},\mathbf{c})) \equiv B(\mathbf{a},\mathbf{b},\mathbf{c}) \bmod p^{e},
\end{equation}
where $e \ge  d + (k-1 -v_p(a_2)) + (k-1-v_p(a_3)) \ge 2k + v_p(a_2+a_3) \ge 2k$, and so the summand of $S_2$ corresponding to $(\mathbf{a},\mathbf{b},\mathbf{c})$ vanishes modulo $p^{2k}$ in this case also. All in all, $S_2 \equiv 0 \bmod p^{2k}$.

Finally, we deal with $S_3$. We have shown that $S_1$ and $S_2$ vanish modulo $p^{2k}$. By \eqref{eq:diff s123}, it suffices to show that $p^{2k} \mid 3S_3$. If $p \neq 3$ then clearly $S_3=0$ since $p \mid (\mathbf{a},\mathbf{b},\mathbf{c})$ forces $(\mathbf{a},\mathbf{b},\mathbf{c})/p \in S(np^{k-1})$. So from now on we suppose that $p=3$, and shall show $3^{2k-1} \mid S_3$.

We first demonstrate that each summand of $S_3$ is divisible by $3^{2k-2}$. As this is vacuous for $k=1$, we assume $k \ge 2$. The condition $(\mathbf{a},\mathbf{b},\mathbf{c}) \notin S(n3^{k-1})$ implies $3 \nmid b_2+2b_3+2c_2+c_3$, which forces one of $\{b_2,b_3,c_2,c_3\}$ not being divisible by $3$, say $3 \nmid b_2$. Next, at least one of $a_2$ and $c_2$ is also indivisible by $3$ (say, $a_2$) since otherwise the condition $n3^{k-1} = a_2+b_2+c_2$ gives a contradiction modulo $3$. The equality $\binom{3^kn}{3\mathbf{a}}\binom{3^kn}{3\mathbf{b}}\binom{3^kn}{3\mathbf{c}} = \frac{3^k n}{3b_2} \frac{3^kn}{3a_2} \binom{3^kn-1}{3b_1,3b_2-1,3b_3}\binom{3^kn-1}{3a_1,3a_2-1, 3a_3} \binom{3^kn}{3\mathbf{c}}$ implies that indeed each summand of $S_3$ is divisible by $3^{2k-2}$. Assuming still that $k \ge 2$, we shall show that 
\begin{equation}\label{eq:mod3}
3 \mid \sum_{\substack{3\cdot (\mathbf{a},\mathbf{b},\mathbf{c}) \in S(n3^k)\\ (\mathbf{a},\mathbf{b},\mathbf{c}) \notin S(n3^{k-1})}} \binom{3^kn}{3\mathbf{a}}\binom{3^kn}{3\mathbf{b}}\binom{3^kn}{3\mathbf{c}} 3^{-(2k-2)},
\end{equation}
finishing the proof for this case. For each $(\mathbf{a},\mathbf{b},\mathbf{c})$ in the summation range we define $\Delta_1 = a_2-a_3 \bmod 3$, $\Delta_2 = b_2-b_3 \bmod 3$ and $\Delta_3 = c_2-c_3 \bmod 3$. The condition $3 \nmid b_2+2b_3+2c_2+c_3$ is equivalent to $\Delta_2 \neq \Delta_3$. Moreover, we have $\Delta_1 + \Delta_2 + \Delta_3 = 0$. This forces $\{ \Delta_1,\Delta_2,\Delta_3\}$ being a permutation of $\{0 \bmod 3,1 \bmod 3 ,2 \bmod 3\}$, and so if $(\mathbf{a},\mathbf{b},\mathbf{c})$ is in the summation range, then $(\mathbf{b},\mathbf{c},\mathbf{a})$ and $(\mathbf{c},\mathbf{a},\mathbf{b})$ are two additional distinct tuples in the range, and evidently cyclic permutations of $(\mathbf{b},\mathbf{c},\mathbf{a})$ contribute the same value to the right-hand side of \eqref{eq:mod3}, proving the congruence \eqref{eq:mod3} holds.

It is left to deal with the case $k=1$, which we deal with differently. We must prove that $B_{3n} \equiv B_n \bmod 9$ for $3 \nmid n$. By the formula $B_n = \sum_{k=0}^{\lfloor n/3 \rfloor} (-1)^k 3^{n-3k}\binom{n}{3k} \binom{3k}{2k} \binom{2k}{k}$, dating back to Zagier's paper \cite{zagier2009}, we have $B_{3n} \equiv (-1)^n \binom{3n}{n,n,n} \bmod 9$ and $B_n \equiv (-1)^{\lfloor \frac{n}{3} \rfloor} 3^{n-3\lfloor \frac{n}{3} \rfloor}\binom{n}{3 \lfloor \frac{n}{3} \rfloor } \binom{3\lfloor \frac{n}{3}\rfloor}{\lfloor \frac{n}{3}\rfloor,\lfloor \frac{n}{3}\rfloor,\lfloor \frac{n}{3}\rfloor} \bmod 9$. For $n=1$, it is easy to check that both expression are $3 \bmod 9$. For $n > 1$ (with $3 \nmid n$), both expressions are $0$ modulo $9$ by e.g. Kummer's theorem \cite{kummer1852}. \qed
\section{Legendrian and hypergeometric sequences}
Zagier \cite{zagier2009} and Almkvist and Zudilin \cite{almkvist2006} have found the following Legendrian and hypergeometric Ap{\'e}ry-like sequences:
\begin{equation}
	\begin{split}
		u_n &= C_{\alpha}^n \sum_{k=0}^{n} (-1)^k \binom{-\alpha}{n-k}^2 \binom{\alpha-1}{k}, \qquad u_n = C_{\alpha}^n \sum_{k=0}^{n} \binom{-\alpha}{n-k}^2 \binom{\alpha-1}{k}^2,\\
		u_n &= C_{\alpha}^n \binom{-\alpha}{n}\binom{\alpha-1}{n}, \qquad u_n = C_{\alpha}^n \binom{-\alpha}{n}\binom{\alpha-1}{n}\binom{2n}{n}
	\end{split}
\end{equation}
for $\alpha \in \{1/2, 1/3,1/4,1/6\}$, where $C_{\alpha} =\alpha^{-3}$ if $\alpha \in \{1/3, 1/4\}$ and $C_{\alpha} = 2\alpha^{-3}$ if $\alpha \in \{1/2, 1/6\}$. It is relatively easy to find constant-term representations for the hypergeometric sequences, several are listed in the OEIS. For the second-order Legendrian sequence $u_n= 16^n \sum_{k=0}^{n} (-1)^k \binom{-1/2}{n-k}^2 \binom{-1/2}{k} =\sum_{k=0}^{n} 4^k \binom{2n-2k}{n-k}^2 \binom{2k}{k}$ we have found and verified the representation $u_n = \CT(\Lambda^n)$ with
\begin{equation}
	\Lambda = (-xy)^{-1}  (x^2 y^2 - 2 x^2 y - 2x y^2 + x^2 - 12xy + y^2 - 2x -2y+ 1).
\end{equation}
The only interior point of $\Lambda$'s Newton polytope is the origin, and $\Lambda$ becomes a `good' Laurent polynomial upon substituting $x^2,y^2$ in place of $x,y$. 
Similarly, for the second-order Legendrian sequence $u_n = 27^n \sum_{k=0}^{n} (-1)^k \binom{-1/3}{n-k}^2 \binom{-2/3}{k}$ we have a representation with
\begin{equation}
	\Lambda = (xy)^{-1}  (x^3 + 3x^2y + 3xy^2 + y^3 - 3x^2 + 21xy - 3y^2 + 3x + 3y - 1)
\end{equation}
which has no interior points other than the origin, and which factorizes upon substituting $x^3,y^3$ for $x,y$. 

For the third-order Legendrian sequence $u_n = 16^n \sum_{k=0}^{n} \binom{-1/2}{n-k}^2 \binom{-1/2}{k}^2=\sum_{k=0}^{n} \binom{2n-2k}{n-k}^2 \binom{2k}{k}^2$, we have found a constant-term representation with 
\begin{equation}
	\Lambda = (xyz)^{-1}  (-x + y + z + 1) (x - y + z + 1)  (x + y - z + 1)     (x + y + z - 1),
\end{equation}
with no interior points other than the origin. It would be interesting to find useful representations for all the Legendrian solutions.
\section{Future directions}
\begin{enumerate}
\item Suppose an integer sequence $\{ u_n\}_{n \ge 0}$ is given in terms of a binomial sum, or as the diagonal of rational function. Is there an algorithm deciding whether $u_n = \CT(\Lambda^n)$ for some Laurent polynomial $\Lambda$, and for finding such a $\Lambda$? Is there an efficient algorithm which produces $\Lambda$ such that the only integral point in its interior is the origin, in case such $\Lambda$ exists?\\Positive answers to these questions would lead to an alternative proof of Theorem~\ref{thm:newton}. The first question is a special case of a question raised by Zagier \cite[p. 769, Q1]{zagier2018}, asking how to recognize that a differential equation has geometric origin. 
\item Fix $d \in \{2,3\}$. Can one classify integer sequences $\{ u_n\}_{n \ge 0}$ such that $y=\sum_{n \ge 0} u_n z^n$ satisfies \eqref{eq:diff2} (if $d=2$) or \eqref{eq:diff3} (if $d=3$), and in addition $u_n = \CT(\Lambda^n)$ for some Laurent polynomial $\Lambda$ in $d$ variables?\\Such a classification might give support to the conjecture that the 15 sporadic sequences that were discovered are the only integral solutions to \eqref{eq:diff2} and \eqref{eq:diff3}.
\item Can one find new Calabi-Yau equations by considering the annihilating operator of $\CTS(\Lambda)$ for various good Laurent polynomials $\Lambda$? Here `good' is used as in \S\ref{sec:newconstant}.
\item Straub generalized the third-order Gauss congruences for the Ap\'ery nmbers $(\gamma)$ as follows. He proved \cite[Thm.~1.2]{straub2014} that for every $\mathbf{n}=(n_1,n_2,n_3,n_4) \in \ZZ^4$, $A(p^r \mathbf{n}) \equiv A(p^{r-1}\mathbf{n}) \bmod p^{3r}$ holds for all $r \ge 1$ and primes $p \ge 5$. Here $A(\mathbf{n})$ is the coefficient of $x_1^{n_1} x_2^{n_2} x_3^{n_3} x_4^{n_4}$ in the Taylor expansion of \eqref{eq:straub gamma}. Since $A(n,n,n,n)=a_n$, this recovers $a_{np^r} \equiv a_{np^{r-1}} \bmod p^{3r}$.

Given a Laurent polynomial $\Lambda \in \ZZ[x_1,x_1^{-1},\ldots,x_d,x_d^{-1}]$, we may consider the rational function $A=1/(1-\prod_{i=1}^{d+1} x_i \Lambda)$ with Taylor coefficients $A(\mathbf{n})$. Generalizing observation \eqref{eq:cts to diag}, note that
\begin{equation}
A(\mathbf{n}) = \CT \big( \Lambda^{n_{d+1}} \prod_{i=1}^{d} x_i^{n_{d+1}-n_i} \big)
\end{equation}
for $\mathbf{n} \in \ZZ^{d+1}$. Thus, a family of congruences $A(p^r \mathbf{n}) \equiv A(p^{r-1}\mathbf{n}) \bmod p^{kr}$, of the form considered by Straub, becomes
\begin{equation}\label{conj:ct_gen}
\CT\bigg( \Lambda^{p^r n} \prod_{i=1}^{d} x_i^{p^r n_i} \bigg) \equiv \CT\bigg( \Lambda^{p^{r-1} n} \prod_{i=1}^{d} x_i^{p^{r-1} n_i} \bigg) \bmod p^{kr}.
\end{equation}
Can one prove \eqref{conj:ct_gen} for the Laurent polynomials in Proposition~\ref{prop:rep}, with $k=2$ or $k=3$?
\end{enumerate}

\section*{Acknowledgments}
In addition to useful comments, corrections and suggestions, I would like to thank Armin Straub for crucial computational help in the proof of Proposition~\ref{prop:rep} for $(\eta)$, and Wadin Zudilin for letting me know of the survey \cite{zagier2018}. I am grateful for the referee's helpful comments regarding the exposition.
\section*{Funding}
 This project has received funding from the European Research Council (ERC) under the European Union's Horizon 2020 research and innovation programme (grant agreement No 851318).

\bibliographystyle{abbrv}
\bibliography{references}

\Addresses

\end{document}